\documentclass{amsart}

\usepackage{graphicx}
\usepackage{amscd}
\usepackage{latexsym}
\usepackage{amssymb}
\usepackage{amsthm}
\usepackage[mathscr]{eucal}
\usepackage{amsmath}

\numberwithin{equation}{section}
\numberwithin{figure}{section}

\theoremstyle{plain}
\newtheorem{theorem}{Theorem}[section]
\newtheorem{corollary}[theorem]{Corollary}
\newtheorem{lemma}[theorem]{Lemma}
\newtheorem{proposition}[theorem]{Proposition}

\theoremstyle{definition}
\newtheorem{definition}[theorem]{Definition}

\theoremstyle{remark}
\newtheorem{remark}[theorem]{Remark}

\begin{document}

\title[Finding a system of essential 2-orbifolds]{Finding a system of essential 2-suborbifolds}

\author{Yoshihiro Takeuchi}
\address{Department of Mathematics,
         Aichi University of Education,
         Igaya, Kariya 448-0001, Japan}
\email{yotake@auecc.aichi-edu.ac.jp}

\author{Misako Yokoyama}
\address{Department of Mathematics,
         Faculty of Science, Shizuoka University,
         Ohya, Suruga-ku, Shizuoka 422-8529, Japan}
\email{smmyoko@ipc.shizuoka.ac.jp}


\begin{abstract}
We make an analogy of Culler-Morgan-Shalen theory. Our main goal is to show that there exists a non-empty system of essential 2-suborbifolds respecting a given splitting of the orbifold fundamental group.
\end{abstract}

\maketitle


\section{Introduction}

For a 3-dimensional manifold $N$, the essential (i.e., incompressible and not boundary parallel) 2-suborbifolds are corresponding to the decompositions of the fundamental group of $N$. If $N$ has an essential and separating 2-subsphere, $\pi_1(N)$ has the free product decomposition which respects its geometric decomposition, and conversely, if $\pi_1(N)$ has a free product decomposition, $N$ has an essential and separating 2-subsphere which realizes its algebraic decomposition. If $N$ has an essential 2-submanifold $S$ which is not a 2-sphere but separating, the above decomposition of $\pi_1(N)$ turns to be an amalgamated free product decomposition, and if $S$ is non-separating, it does to be an HNN extension decomposition. Moreover, \cite{C-S} proved the theorem that if $N$ acts on a simplicial tree nontrivially, $N$ has a non-empty system of essential 2-submanifolds which respects that action.

A similar approach should be considered for 3-orbifolds. If a 3-obifold $M$ has an essential 2-suborbifold, it is clear that the orbifold fundamnetal group $\pi_1(M)$ has an amalgamated free product decomposition or an HNN extension decomposition. In \cite{fes} (respectively, \cite{kune}) we found an essential non-spherical (respectively, spherical) 2-suborbifold realizing a given algebraic decomposition of the orbifold fundamental group of $M$.

In the present paper we show the following:

\begin{theorem}\label{copy-main1}
Let $M$ be a good, compact, connected, orientable 3-orbifold without non-separating spherical 2-orbifolds. We assume that the fundamental group of each prime component of $M$ is infinite. Suppose that $\pi_1(M)$ has a nontrivial finite splitting. Then there exists a non-empty system of essential 2-suborbifolds $S_1,\dots,S_n\subset M$ such that for each component $Q$ of $\displaystyle M-\bigcup_{i=1}^nS_i$, $\pi_1(Q)$ is contained in a vertex group.
\end{theorem}

Boileau, Maillot and Porti showed a related result in \cite[Proposition 7.16]{BMP}, where they treat with the fundamental group of the complement of the set of singular points of a 3-orbifold.

We summarize the contents of the present paper. In Section 2, 3 and 4, we review on the actions on a tree, 3-orbifolds, and OISIBO's (orbifold identified spaces identified along ballic orbifolds) respectively. In Section 5, we prepare an orbifold composition, which is used in Main Theorem as the target space of a b-continuous map. In Section 6, we prove Main Theorem.

\section{Preliminaries on the actions on a tree}

Throughout this present paper any orbifold is assumed to be good, that is, it is covered by a manifold, and assumed to be compact, connected and orientable unless otherwise stated.

In \cite{Se}, some fixed point theorems about group actions on trees are proved. Here we use the following restricted forms of them.

Let $T$ be a simplicial tree, i.e., a connected and simply connected 1-complex, and $G$ a group simplicially acting on $T$.

For $g\in G$, $g$ is called to have an {\it edge inversion\/} if there exists an edge $E$ such that $g(E)=E$ and $g|E$ is orientation reversing.

The action is called {\it trivial} if a vertex of $T$ is fixed by $\Gamma$.

\begin{proposition}\label{tree-new-Prop}
Let $g$ be an element of $G$ with finite order. If $g$ acts on $T$ without edge inversions, then there exists a vertex $p$ of $T$ such that $g(p)=p$.
\end{proposition}

\begin{proposition}\label{fes-Prop2.1}
Let $p_1$, $p_2\in T$ be fixed points of $g\in G$ and $\ell$ the unique simple path from $p_1$ to $p_2$. Then any vertex and edge on $\ell$ are fixed by $g$.
\end{proposition}

Let $n\geq 1$ be an integer. Put
$$G_n= \; \langle \; a_1,\cdots,a_n|\; a_1^{\alpha_1}=\cdots=a_n^{\alpha_n}=(a_ia_j)^{\beta_{i,j}}=1,\; 1\leq i<j\leq n\;\rangle $$
where $\alpha_i,\beta_{i,j}\geq 2$ are integers.

\begin{lemma}\label{fes-Lem2.2}
If $G_n$ acts on $T$ without edge inversions, then $T$ has a fixed vertex of $G_n$ action.
\end{lemma}

\section{Preliminaries on orbifolds}

\begin{definition}\label{contmap-def}
Let $M=(\tilde{M},p,|M|)$, $N=(\tilde{N},q,|N|)$ be orbifolds. A {\it continuous map} $f:M\rightarrow N$ is a pair $(|f|,\tilde{f})$ of continuous maps $|f|:|M|\rightarrow |N|$ and $\tilde{f}:\tilde{M}\rightarrow \tilde{N}$ which satisfies the following:
\begin{equation}\label{ofd-diag}
\CD
\tilde{M} @>\tilde{f}>> \tilde{N} \\
@VpVV @VVqV \\
|M| @>>|f|> |N| \\
\endCD
\end{equation}
\begin{enumerate}
\item[{\bf (i)}] $|f|\circ p=q\circ \tilde{f}$,
\item[{\bf (ii)}] For each $\sigma \in$ Aut$(\tilde{M},p)$ there exists an element $\tau\in$ Aut$(\tilde{N},q)$ such that $\tilde{f}\circ\sigma=\tau\circ\tilde{f}$.
\end{enumerate}
A continuous map $f:M\rightarrow N$ is {\it b-continuous\/} if there exists a point $x\in|M|-\Sigma M$ such that $|f|(x)\in|N|-\Sigma N$. It was called an orbi-map in \cite{finite,least}, etc. A b-continuous map induces a homomorphism between the fundamental groups and local fundamental groups of orbifolds, see \cite[Lemma 3.13]{t-homo} and \cite{finite}, where the points $x$ and $|f|(x)$ in the above could be base points of the fundamental groups of $M$ and $N$, respectively. The notion of (b-)continuous maps between orbifolds is naturally generalized for those between OISIBO's and orbifold compositions in Sections \ref{OISIBO-sec} and \ref{orbifoldcomposition-sec}.

A b-continuous map $f:M\rightarrow N$ is an {\it embedding} if $f(M)$ is a suborbifold of $N$ and $f:M\rightarrow f(M)$ is an isomorphism of orbifolds.
\end{definition}

For other terminologies, see \cite{finite}.

\section{Preliminaries on OISIBO's}\label{OISIBO-sec}

\begin{definition}\label{OISIBO}
Let $I$, $J$ be countable sets, $X_i$ ($i\in I$) n-orbifolds, and $B_j$ ($j\in J$) ballic n-orbifolds. Let $f_j^{\varepsilon}:B_j\rightarrow X_{i(j,\varepsilon)}$ be embeddings (as orbifolds) such that $f^{\varepsilon}_{j}(B_j)\subset$ Int $X_{i(j,\varepsilon)}$ and $f_j^{\varepsilon}(B_j)$ are mutually disjoint, where $j\in J$, $i(j,\varepsilon)\in I$, $\varepsilon =0,1$. Then we call $X=(X_i,B_j,f_j^{\varepsilon})_{i\in I,j\in J,\varepsilon=0,1}$ an {\it $n$-orbifold identified space identified along ballic orbifolds} ($n$-OISIBO). The maps $f_j^0\circ (f_j^1)^{-1}$ and their inverses are called the {\it identifying maps} of $X$. Each $X_i$, $B_j$ are called a {\it particle of} $X$, and an {\it identifying ballic orbifold}, respectively. We define the equivalence relation $\sim$ in $\coprod_{i\in I,j\in J}(|X_i|\cup |B_j|)$ to be generated by
\begin{equation}
y\sim|f|^{\varepsilon}_j(y),\quad\varepsilon=0,1,\quad y\in|B_j|,\quad j\in J.
\end{equation}
We call the identified space $\coprod_{i\in I,j\in J}(|X_i|\cup|B_j|)/\sim$ the {\it underlying space of} $X$, denoted by $|X|$, and call the identified space $\{(\cup_{i\in I}\Sigma X_i)\cup(\cup_{j\in J}\Sigma B_j)\}/\sim$ the {\it singular set of} $X$, denoted by $\Sigma X$.
\end{definition}

\begin{definition}\label{OISIBOiso}
Let $X=(X_i,B_j,f_j^{\varepsilon})_{i\in I,j\in J,\varepsilon=0,1}$ and $X'=(X_k',B_{\ell}',g_{\ell}^{\varepsilon})_{k\in K,\ell\in L,\varepsilon=0,1}$ be OISIBO's. We say that $X$ and $X'$ are {\it isomorphic} if there exists a set of maps $\{\varphi_i,\psi_j\}_{i\in I,j\in J}$ and bijections $\eta:I\to K$, $\xi:J\to L$ such that the following (i) and (ii) hold:
\begin{enumerate}
\item[{\bf (i)}] For each $i\in I$, $\varphi_i$ is an isomorphism (of orbifolds) from $X_i$ to $X_{\eta(i)}'$, and for each $j\in J$, $\psi_j$ is an isomorphism (of orbifolds) from $B_j$ to $B_{\xi(j)}'$.
\item[{\bf (ii)}] For each $j\in J$, and $\varepsilon =0,1$, $\varphi_{i(j,\varepsilon)}\circ f_j^{\varepsilon}=g_{\xi(j)}^{\varepsilon}\circ\psi_j$.
\end{enumerate}

The system $h=(\{\varphi_i,\psi_j\}_{i\in I,j\in J},\eta,\xi)$ is called an {\it isomorphism from} $X$ {\it to} $X'$.
\end{definition}

\begin{definition}\label{OISIBOcovering}
Let $X=(X_k,B_{\ell},f_{\ell}^{\varepsilon})_{k\in K,\ell\in L,\varepsilon=0,1}$ and $X'=(X_i',B_j',{f'}_j^{\varepsilon})_{i\in I,j\in J,\varepsilon=0,1}$ be OISIBO's. We say that $X'$ is a {\it covering of} $X$ if there exists a set of maps $\{\varphi_i,\psi_j\}_{i\in I,j\in J}$ and surjections $\eta:I\rightarrow K$, $\xi:J\rightarrow L$ such that the following (i) and (ii) hold:
\begin{enumerate}
\item[{\bf (i)}] Each $\varphi_i$ is a covering map (of orbifolds) from $X'_i$ to $X_{\eta(i)}$, where $\eta(i)\in K$, and each $\psi_j$ is a covering map (of orbifolds) from $B_j'$ to $B_{\xi(j)}$, where $\xi(j)\in L$.
\item[{\bf (ii)}] For each $j\in J$ and $\varepsilon =0,1$, $\varphi_{i(j,\varepsilon)}\circ {f'}_j^{\varepsilon}=f_{\xi(j)}^{\varepsilon}\circ\psi_j$.
\end{enumerate}
Note that the continuous map $|p|:|X'|\to|X|$ naturally induced by $\{\varphi_i,\psi_j\}_{i\in I,j\in J}$ is surjective, and induces the usual covering map from $|X'|-|p|^{-1}(\Sigma X)$ to $|X|-\Sigma X$. We call the system $p=(|p|,\{\varphi_i,\psi_j\}_{i\in I,j\in J})$
a {\it covering map from} $X'$ {\it to} $X$.
\end{definition}

\begin{definition}\label{OISIBOunicov}
Let $\tilde{X}$, $X$ be OISIBO's, and $p:\tilde{X}\to X$ a covering. We call $p$ a {\it universal covering} if for any covering $p':X'\to X$, there exists a covering $q:\tilde{X}\to X'$  such that $p=p'\circ q$. As the usual covering theory, for any OISIBO $X$, there exists a unique universal covering $p:\tilde{X}\to X$.
\end{definition}

\begin{definition}\label{OISIBOdeck}
Let $X'$, $X$ be OISIBO's, and $p:X'\to X$ a covering. We define the {\it deck transformation group} Aut$(X',p)$ {\it of} $p$ by
\begin{equation}
\text{Aut}(X',p)=\{ h:X'\to X'\;|\;h \hbox{ is an isomorphism such that } p\circ h=p\}.
\end{equation}
\end{definition}

We sometimes denote an OISIBO $X$ by $(\tilde{X},p,|X|)$, where $p:\tilde{X}\rightarrow X$ is the universal covering and $|X|$ is the underlying space of $X$. Any orbifold is considered as a special case of an OISIBO.

\begin{definition}\label{OISIBOcontmap-def2}
Let $X=(\tilde{X},p,|X|)$, $Y=(\tilde{Y},q,|Y|)$ be OISIBO's. A {\it continuous map} $f:X\rightarrow Y$ is a pair $(|f|,\tilde{f})$ of continuous maps $|f|:|X|\rightarrow |Y|$ and $\tilde{f}:\tilde{X}\rightarrow \tilde{Y}$ which satisfies the same property as (i) and (ii) in Definition \ref{contmap-def}. A continuous map $f:X\rightarrow Y$ is {\it b-continuous\/} if there exists a point $x\in|X|-\Sigma X$ such that $|f|(x)\in|Y|-\Sigma Y$.
\end{definition}

We define a {\it homotopy\/} of OISIBO's by using of continuous maps of OISIBO's as the usual homotopy. If the continuous maps at $0$ {\it and\/} $1$ levels of the homotopy are b-continuous, this homotopy is called a b-{\it homotopy}. See \cite{t-homo}.

We define a {\it path\/} in an OISIBO $X$ by using of a b-continuous map $\alpha=(|\alpha|,\tilde{\alpha}):[0,1]\rightarrow X$ with $|\alpha|(0)\in|X|-\Sigma X$. If a path $\alpha$ in $X$ satisfies that $|\alpha|(0)=|\alpha|(1)$, it is called a {\it loop\/} in $X$.

By using of loops in an OISIBO $X$, we define the fundamental group of $X$ as the usual theory. A b-continuous map $f:X\to Y$ between OISIBO's $X$ and $Y$ induces a homomorphism between the fundamental groups and local fundamental groups of $X$ and $Y$, where the points $x\in |X|-\Sigma X$ and $|f|(x)\in|Y|-\Sigma Y$ in the definition of b-continuous map are the base points of the fundamental groups of $X$ and $Y$, respectively.

As usual covering theory, various similar results holds such as the following:

\begin{proposition}
Let $X$ be an OISIBO and let $x$, $y$ be any two points of $|X|-\Sigma X$. Then the fundamental groups $\pi_1(X,x)$ and $\pi_1(X,y)$ are isomorphic.
\end{proposition}

We often denote $\pi_1(X,x_0)$ by $\pi_1(X)$ dropping a base point if not necessary.

\begin{proposition}
Let $X$ be an OISIBO and $p:\tilde{X}\rightarrow X$ the universal covering of $X$. Then the fundamental group $\pi_1(X)$ is isomorphic to the deck transformation group {\rm Aut}$(\tilde{X},p)$.
\end{proposition}

\begin{proposition}
Let $X$ be an OISIBO. For each subgroup $H$ of $\pi_1(X)$, there exists a covering $p:\tilde{X}\rightarrow X$ such that the OISIBO $\tilde{X}$ has a fundamental group which is isomorphic to $H$.
\end{proposition}

\section{Orbifold compositions}\label{orbifoldcomposition-sec}

\begin{definition}\label{composition-def}
Let $I$, $J$ be countable sets, $X_i$ ($i\in I$) and $Y_j$ ($j\in J$) be n-OISIBO's. Let $f_j^{\varepsilon}:Y_j\times \varepsilon \rightarrow X_{i(j,\varepsilon)}$ be b-continuous-maps, $f^{\varepsilon}_j=(|f^{\varepsilon}_j|,\tilde{f}^{\varepsilon}_j)$, such that $(f^{\varepsilon}_{j})_*$ are monic, where $j\in J$, $i(j,\varepsilon)\in I$, $\varepsilon =0,1$. Then we call $X=(X_i,Y_j\times[0,1],f_j^{\varepsilon})_{i\in I,j\in J,\varepsilon=0,1}$ an {\it n-dimensional orbifold composition {\rm (}of type III {\rm )}}. The maps $f_j^{\varepsilon}$ are called the {\it attaching maps} of $X$, which may have intersections and self-intersections. Each $X_i$, $Y_j\times[0,1]$ is called a {\it component of} $X$. The equivalence relation $\sim$ in $\coprod_{i\in I,j\in J}(|X_i|\cup (|Y_j|\times[0,1]))$ is defined to be generated by
\begin{equation}
(y,\varepsilon)\sim|f^{\varepsilon}_j|(y,\varepsilon),\quad\varepsilon=0,1,\quad y\in|Y_j|,\quad j\in J.
\end{equation}
We call the identified space $\coprod_{i\in I,j\in J}\bigl(|X_i|\cup(|Y_j|\times[0,1])\bigr)/\sim$ the {\it underlying space of} $X$, denoted by $|X|$, and call the identified space $\{(\cup_{i\in I}\Sigma X_i)\cup(\cup_{j\in J}\Sigma(Y_j\times[0,1]))\}/\sim$ the {\it singular set of} $X$, denoted by $\Sigma X$.
\end{definition}

In the definition of an orbifold composition, each $(f^{\varepsilon}_j)_*$ is monic, so that we can obtain the unique lift of any path $\ell$ in $X$ such that $\ell[0,1)\cap\Sigma X=\emptyset$.

We define a covering of an orbifold composition as similar as that of an OISIBO. We sometimes denote an orbifold composition $X$ by $(\tilde{X},p,|X|)$, where $p:\tilde{X}\rightarrow X$ is the universal covering and $|X|$ is the underlying space of $X$. Any orbifold and any OISIBO are considered as special cases of an orbifold composition.

\begin{definition}\label{compositioncontmap-def2}
Let $X=(\tilde{X},p,|X|)$, $Y=(\tilde{Y},q,|Y|)$ be orbifold compositions. A {\it continuous map} $f:X\rightarrow Y$ is a pair $(|f|,\tilde{f})$ of continuous maps $|f|:|X|\rightarrow |Y|$ and $\tilde{f}:\tilde{X}\rightarrow \tilde{Y}$ which satisfies the same property as (i) and (ii) in Definition \ref{contmap-def}. A continuous map $f:X\rightarrow Y$ is {\it b-continuous\/} if there exists a point $x\in|X|-\Sigma X$ such that $|f|(x)\in|Y|-\Sigma Y$.
\end{definition}

By using of loops in an orbifold composition $X$, we define the fundamental group of $X$ as the usual theory. A b-continuous map $f:X\to Y$ between orbifold compositions $X$ and $Y$ induces a homomorphism between the fundamental groups and local fundamental groups of $X$ and $Y$, where the points $x\in |X|-\Sigma X$ and $|f|(x)\in|Y|-\Sigma Y$ in the definition of b-continuous map are the base points of the fundamental groups of $X$ and $Y$, respectively.

\begin{definition}\label{oix}
Let $X$ be an orbifold composition. Define $O_i(X)$, $i=1,2,3$ as follows:
\begin{align*}
O_1(X) &= \{ f:S^1\to X \hbox{, a b-continuous map } | \; [f] \hbox{ is of finite order } (\neq 1) \hbox{ in } \pi_1(X)\}, \\
O_2(X) &= \{ f:S\to X \hbox{, a b-continuous map } | \; S \hbox{ is a spherical 2-orbifold} \}, \\
O_3(X) &= \{ f:\mathcal{D}B\to X \hbox{, a b-continuous map } | \; \mathcal{D}B \hbox{ is the double of a ballic 3-orbifold } B \}.
\end{align*}

A b-continuous map $f:S^1\rightarrow X\in O_1(X)$ is {\it trivial} if there exists a b-continuous map $g$ from a discal 2-orbifold $D$ to $X$ such that $g|\partial D=f$ and the index of $D$ equals to the order of $[f]$. $O_1(X)$ is {\it trivial} if every element of $O_1(X)$ is trivial. We call $f:S\to X\in O_2(X)$ {\it trivial} if there exists a b-continuous map $g:c*S\rightarrow X$ such that $g|S=f$, where $c*S$ is the cone on $S$. $O_2(X)$ is {\it trivial} if every element of $O_2(X)$ is trivial. We define the triviality of $O_3(X)$ similarly.

Note that if $O_i(X)$ is trivial, then any covering $\tilde{X}$ of $X$ inherits the triviality.
\end{definition}

\begin{remark}
Let $X$ be an orbifold composition, and $\tilde{X}$ the universal cover of $X$. If $O_2(X)$ is trivial, then $\pi_2(\tilde{X})=0$.
\end{remark}

\begin{proposition}\label{kune-Prop5.3'}
Let $M$ be a compact 3-orbifold, and $X$ an orbifold composition. If $O_1(X)$ and $O_2(X)$ are trivial, then for any homomorphism $\varphi:\pi_1(M)\rightarrow\pi_1(X)$, there exists a b-continuous map $f:M\rightarrow X$ such that $f_*=\varphi$.
\end{proposition}

\begin{proof}
Let $M_0=|M|-$ Int $U(\Sigma M)$, where $U(\Sigma M)$ is the small regular neighborhood of $\Sigma M$. We can construct a b-continuous map from the $1$-skelton of $M_0$ to $X$ associated with $\varphi$. Since $O_1(X)$ and $O_2(X)$ are trivial, this b-continuous map is extendable to $M_0$ and furthermore to $M$, that is, we have obtained the desired b-continuous map.
\end{proof}

\begin{proposition}\label{kune-Prop5.5'}
Let $X$ be an orbifold composition, and $f:S^1\to X$ a b-continuous map. If$\;$ {\rm Fix}$([f]_A)\ne\phi$, then $f$ is extendable to a b-continuous map from a discal 2-orbifold $D$ to $X$ where $D=D^2(n)$, and $n$ is the order of $[f]_A$.
\end{proposition}

\begin{proof}
Let $q:D^2\to D$ be the universal covering. Choose a point $x\in$ Fix$([f]_A)$. We can construct the structure map of the desired b-continuous map by mapping the cone point of $D^2$ to $x$ and performing the skeletonwise and equivariant extension.
\end{proof}

Let $S$ be a spherical 2-orbifold and let $q:(\tilde{S},\tilde{x}_0)\to (S,x_0)$, $x_0\not\in \Sigma S$, be the universal covering. Let $\tau$ be  an element of $\pi_1(S,x_0)$, $\tilde{x}_{\tau}$ one of the two points of Fix$(\tau_A)$, and $x_{\tau}=q(\tilde{x}_{\tau})$. By the symbol $\mu(x_{\tau})$, we shall mean the local normal loop around $x_{\tau}$. Let $\ell$ be a path in $|S|-\Sigma S$ from $\mu(x_{\tau})(0)$ to $x_0$ such that $\tau=([\ell^{-1}\cdot\mu(x_{\tau})\cdot\ell])^k$, $k\in\mathbb{Z}$

\begin{proposition}\label{kune-Prop5.6'}
Let $S$ be a spherical 2-orbifold, $X$ an orbifold composition, and $f:S\to X$ a b-continuous map. Suppose that there exists a point $\tilde{d}\in$ {\rm Fix}$(f_*\pi_1(S))_A$, and for any $\tau\in\pi_1(S,x_0)$ there exists an interval $\ell_{\sigma}$ including $\tilde{d}$ and $\tilde{f}(\tilde{x}_{\tau})$ which is fixed by $\sigma_A$, where $\sigma=f_*(\tau)$. If $\pi_2$ of the universal cover $\tilde{X}$ of $X$ is $0$, then $f$ is extendable to a b-continuous map from the cone on $S$ to $X$.
\end{proposition}

\begin{proof}
The proof is similar to that of \cite[Proposition 5.6]{kune}.
\end{proof}

\begin{proposition}\label{kune-Prop5.7'}
Let $\mathcal{D}B$ be the double of a ballic 3-orbifold $B$, $X$ an orbifold composition, and $f:\mathcal{D}B\to X$ a b-continuous map. Suppose that {\rm Fix}$(f_*\pi_1(\partial B))_A$ is connected, and for $\tau\in\pi_1(\partial B,x_0)$, $\pi_1(${\rm Fix}$(f_*((\tau))_A))=1$ and there exists an interval $\ell_{\sigma}$ including $\tilde{d}$ and $\tilde{f}(\tilde{x}_{\tau})$ which is fixed by $\sigma_A$, where $\sigma=f_*(\tau)$. If $\pi_2$ and $\pi_3$ of the universal cover $\tilde{X}$ of $X$ is $0$, then $f$ is extendable to a b-continuous map from the cone on $\mathcal{D}B$ to $X$.
\end{proposition}

\begin{proof}
The proof is similar to that of \cite[Proposition 5.7]{kune}.
\end{proof}

\begin{proposition}\label{kune-Prop5.8}\cite[Proposition 5.8]{kune}
Let $X$ be a 3-OISIBO whose particles are irreducible. Let $p:\tilde{X}\to X$ be the universal covering and $\sigma\in$ {\rm Aut}$(\tilde{X},p)$ be any  nontrivial element of finite order. Suppose that each particle of $\tilde{X}$ is non-compact. Then the following holds:
\begin{enumerate}
\item[{\bf (i)}] {\rm Fix}$(\sigma )\ne\phi$ and is homeomorphic to a tree.
\item[{\bf (ii)}] $O_1(X)$ is trivial.
\end{enumerate}
\end{proposition}

Since we assume that any orbifold is orientable, the restriction of $\sigma$ to each particle is orientation preserving, and each identifying ballic orbifold is orientable.

\begin{proposition}\label{kune-Prop5.9}\cite[Proposition 5.9]{kune}
Let $X$ be a 3-OISIBO, each particle of which is irreducible, and $p:\tilde{X}\to X$ the universal covering. Let $G$ be any subgroup of {\rm Aut}$(\tilde{X},p)$, which is isomorphic to the orbifold fundamental group of a spherical 2-orbifold $S$. Suppose that each particle of $\tilde{X}$ is non-compact. Then the following holds:
\begin{enumerate}
\item[{\bf (i)}] {\rm Fix}$(G)$ is either a point or a tree.
\item[{\bf (ii)}] $\pi_2(\tilde{X})=\pi_3(\tilde{X})=0$.
\item[{\bf (iii)}] $O_i(X)$'s are trivial, $i=1,2,3$.
\end{enumerate}
\end{proposition}

\begin{proposition}\label{ofd-eq-prop}
Let $X=(\tilde{X},p,|X|)$, $Y=(\tilde{Y},q,|Y|)$ be orbifold compositions, and $f=(|f|,\tilde{f}):X\rightarrow Y$ a b-continuous map. Then for $[\alpha]\in\pi_1(X,x)$,
\begin{equation}
\tilde{f}\circ[\alpha]_A=(f_*([\alpha]))_A\circ\tilde{f}.
\end{equation}
\end{proposition}

\begin{proof}
Let $\tilde{x}\in p^{-1}(x)$ be the base point of $\tilde{X}$. Note that $[\alpha]_A$ is characterized as the element of Aut$(\tilde{X},p)$ which transforms $\tilde{\alpha}(0)=\tilde{x}$ to $\tilde{\alpha}(1)$, and $(f_*([\alpha]))_A=[f\circ\alpha]_A$ is characterized as the element of Aut$(\tilde{Y},q)$ which transforms $\tilde{f}(\tilde{\alpha}(0))$ to $\tilde{f}(\tilde{\alpha}(1))$. By the definition of b-continuous map, there exists an element $\tau\in$ Aut$(\tilde{Y},q)$ such that $\tilde{f}\circ[\alpha]_A=\tau\circ\tilde{f}$. On the other hand, $\tau(\tilde{f}(\tilde{\alpha}(0)))=(\tau\circ\tilde{f})(\tilde{\alpha}(0))=(\tilde{f}\circ[\alpha]_A)(\tilde{\alpha}(0))=\tilde{f}(\tilde{\alpha}(1))$. Hence $\tau=(f_*([\alpha]))_A$.
\end{proof}

\begin{proposition}\label{Oitrivial-prop}
Let $X=(X_i,Y_j\times[0,1],f^j_{\varepsilon})_{i\in I,j\in J,\varepsilon=0,1}$ be an orbifold composition, where each particle of each $X_i$ and $Y_j$ is an orientable and irreducible 3-orbifold whose universal covering is noncompact. Then $O_i(X)$'s are trivial, $i=1,2,3$.
\end{proposition}

\begin{proof}
Let $p:\tilde{X}\to X$ be the universal covering. From the uniqueness of the universal covering, we may assume that $\tilde{X}$ is the orbifold composition constructed in a similar method described in  \cite{fes}.

{\it Claim}:\quad Let $G$ be any subgroup of Aut$(\tilde{X},p)$, which is isomorphic to the fundamental group of a spherical 2-orbifold. Consider the associated 1-complex $\mathcal{C}(\tilde{X})$ of $\tilde{X}$. Then, there exists a vertex OISIBO $\tilde{Z}$ of $\tilde{X}$ with respect to $\mathcal{C}(\tilde{X})$ such that $G(\tilde{Z})=\tilde{Z}$.

Indeed, $G$ is finite and acts on the tree $\mathcal{C}(\tilde{X})$ without edge inversions. By Lemma \ref{fes-Lem2.2} we have the claim.

Take any element $f\in O_1(X)$. By the claim, there exists an OISIBO $\tilde{Z}$ of $\tilde{X}$ such that $[f]_A(\tilde{Z})=\tilde{Z}$. Then by Proposition \ref{kune-Prop5.8}, Fix$([f]_A)\neq\emptyset$ in $\tilde{Z}$. Thus Fix$([f]_A)\neq\emptyset$ in $\tilde{X}$. By using of Proposition \ref{kune-Prop5.5'}, $f$ is extendable to a b-continuous map from $D^2(n)$ to $X$.

Take any element $f\in O_2(X)$, $f:S\to X$. Let $q:\tilde{S}\to S$ be the universal covering and $\tilde{f}:\tilde{S}\to\tilde{X}$ the structure map of $f$. Let $B=c*S$ be the cone on $S$ and $c$ the cone point of $B$. Let $\bar{q}:\tilde{B}=\tilde{c}*\tilde{S}\to B$ be the universal covering, $\tilde{c}=\bar{q}^{-1}(c)$ and $\bar{q}(t\tilde{x}+(1-t)\tilde{c})=tq(\tilde{x})+(1-t)c$, $\tilde{x}\in\tilde{S}$. Note that $f_*\pi_1(S)$ is isomorphic to a spherical 2-orbifold group. By the claim, there exists a vertex OISIBO $\tilde{Z}$ of $\tilde{X}$ such that $(f_*\pi_1(S))(\tilde{Z})=\tilde{Z}$. By Proposition \ref{kune-Prop5.9}, Fix$(f_*\pi_1(S))_A\neq\emptyset$. Thus there exists a point $\tilde{d}\in\tilde{Z}$ such that $(f_*\pi_1(S))_A\tilde{d}=\tilde{d}$.

Choose any $\tau\in\pi_1(S)$. We put $\sigma=f_*(\tau)$. Since $\sigma_A\in(f_*\pi_1(S))_A$, $\sigma_A(\tilde{d})=\tilde{d}$. Moreover, by the fact $\tau_A(\tilde{x}_{\tau})=\tilde{x}_{\tau}$ and Proposition \ref{ofd-eq-prop}, $\sigma_A(\tilde{f}(\tilde{x}_{\tau}))=f_*(\tau)_A(\tilde{f}(\tilde{x}_{\tau}))=\tilde{f}\circ\tau_A(\tilde{x}_{\tau})=\tilde{f}(\tilde{x}_{\tau})$.  

Let $\tilde{Z}_1$ be the OISIBO in which $\tilde{f}(\tilde{x}_{\tau})$ is included and let $\tilde{Z}_k$ be the OISIBO in which $\tilde{d}$ is included. Since $\sigma_A$ acts on the tree ${\mathcal C}(\tilde{X})$, $\tilde{Z}_1$ and $\tilde{Z}_k$ are invariant by $\sigma_A$. In addition, since $\sigma_A$ is of finite order, we can apply Proposition \ref{fes-Prop2.1} and get that any vertex OISIBO $\tilde{Z}_i$ and any edge OISIBO $\tilde{Z}_j$ between $\tilde{Z}_1$ and $\tilde{Z}_k$ are invariant by $\sigma_A$. By Proposition \ref{kune-Prop5.8}, for each $\tilde{Z}_i$ and each $\tilde{Z}_j$, Fix $\sigma_A|\tilde{Z}_i$ is a tree and Fix ($\sigma_A|\tilde{Z}_j$) is $(\mbox{a tree} \times[0,1])$ .

Note that since the structure map of each attaching map from $\tilde{Z}_{i_0}$ to $\tilde{Z}_{i_1}$ and the restriction of $\sigma_A$ to $\tilde{Z}_{i_0}$ and $\tilde{Z}_{i_1}$ commute, any  point of Fix $(\sigma_A|\tilde{Z}_{i_0})$ is mapped to a point of Fix $(\sigma_A|\tilde{Z}_{i_1})$. Hence Fix $\sigma_A$ in $\tilde{X}$ is connected. Thus we can find an interval which is fixed by $\sigma_A$, and connecting $\tilde{f}(\tilde{x}_{\tau})$ and $\tilde{d}$. 

We will show that $\pi_2(\tilde{X})=0$. Take any continuous map $\varphi:S^2\rightarrow\tilde{X}$. Since $\varphi(S^2)$ is compact, there exists a connected and compact subset $P$ of $\tilde{X}$ which contains a finite number of OISIBO's $\tilde{Z}_i$ of $\tilde{X}$ such that $\varphi(S^2)\subset P$. By Proposition \ref{kune-Prop5.9}(ii), for each  $\tilde{Z}_i$, $\pi_2(\tilde{Z}_i)=H_2(\tilde{Z}_i)=0$.  Dividing $P=P_1\cup P_2$ such that $P_1$ consists of $k$ vertex OISIBO's, $(k-1)$ edge OISIBO's, and $\tilde{Y}\times[0,\frac{1}{\; 2\;}+\varepsilon]$ and $P_2$ consists of a vertex OISIBO and $\tilde{Y}\times[\frac{1}{\; 2\;}-\varepsilon,1]$, $P_1\cap P_2=\tilde{Y}\times[\frac{1}{\; 2\;}-\varepsilon,\frac{1}{\; 2\;}+\varepsilon]$ where $\tilde{Y}\times[0,1]$ is an edge OISIBO of $\tilde{X}$, we can show $H_2(P)=0$ by the induction on $k$ and Mayer-Vietoris exact sequences. Thus $\pi_2(P)=0$, which gives the fact $[\varphi]=0$ in $\pi_2(P)$, so is in $\pi_2(\tilde{X})$. Then the triviality of $f$ follows from Proposition \ref{kune-Prop5.6'}.

The triviality of $O_3(X)$ is derived from showing $\pi_3(\tilde{X})=0$ and Proposition \ref{kune-Prop5.7'} in the similar mannerr.
\end{proof}

Let $M$ be a 3-orbifold, and $X$ an orbifold composition. We say that two b-continuous maps $f,g:M\rightarrow X$ are {\it C-equivalent\/} if there is a sequence of b-continious maps $f=f_0,f_1,\dots,f_n=g$ from $M$ to $X$ with either $f_i$ is b-homotopic to $f_{i-1}$ or $f_i$ agrees with $f_{i-1}$ on $M-B$ for a certain ballic 3-orbifold $B\subset M$ with $B\cap\partial M$ a discal orbifold or $|B|\cap|\partial M|=\emptyset$. Note that C-equivalent b-continuous maps induce the same homomorphisms $\pi_1(M)\rightarrow\pi_1(X)$ up to choices of base points and inner automorphisms.

\begin{remark}\label{c-eq-rem}
Let $f$, $g$ be C-equivalent maps from a 3-orbifold $M$ to an orbifold composition $X$. If $O_i(X)$'s are trivial, $i=2,3$, then $f$ and $g$ are b-homotopic.
\end{remark}

\begin{lemma}\label{Finita-lem5.4'}
Let $M$ be an orbifold, and $X$ an orbifold-composition. Let $p:\tilde{M}\to M$ and $q:\tilde{X}\to X$ be the universal coverings. Suppose dim $\tilde{M}=n$ and $\pi_{n-1}(|\tilde{X}|)=0$. Let $\tilde{g}:|\tilde{M}|\to|\tilde{X}|$ be a continuous map which satisfies the condition that there exists a homomorphism $\varphi:\mbox{Aut}(\tilde{M},p)\to\;\mbox{Aut}(\tilde{X},q)$ such that for each $\sigma\in\;\mbox{Aut}(\tilde{M},p)$, $\tilde{g}\circ\sigma=\varphi(\sigma)\circ\tilde{g}$. Then there exists a continuous map $\tilde{f}:|\tilde{M}|\to|\tilde{X}|$ which satisfies the following:
\begin{enumerate}
\item[(1)] There exists a point $\tilde{x}\in|\tilde{M}|-p^{-1}(\Sigma M)$ such that $\tilde{f}(\tilde{x})\in|\tilde{X}|-q^{-1}(\Sigma X)$.
\item[(2)] There exists an $n$-ball $B^n\subset |\tilde{M}|-p^{-1}(\Sigma M)$ such that $B^n\cap\sigma(B^n)=\emptyset$ for each $\sigma\in\;\mbox{Aut}(\tilde{M},p)$, $\sigma\neq\;\mbox{id}$, and
\begin{equation}
\tilde{f}|
\left(
\begin{array}{l}
\displaystyle|\tilde{M}|-\bigcup_{\sigma\in\;\mbox{Aut}(\tilde{M},p)}\sigma(B^n)
\end{array}\right)
=\tilde{g}|
\left(
\begin{array}{l}
\displaystyle|\tilde{M}|-\bigcup_{\sigma\in\;\mbox{Aut}(\tilde{M},p)}\sigma(B^n)
\end{array}\right)
.
\end{equation}
\item[(3)] For each $\sigma\in\mbox{Aut}(\tilde{M},p)$, $\tilde{f}\circ\sigma=\varphi(\sigma)\circ\tilde{f}$.
\end{enumerate}
\end{lemma}

\begin{proof}
The proof is similar to that of \cite[Lemma 5.4]{finite}.
\end{proof}

\begin{theorem}[Transversality Theorem]\label{trans}
Let $M$ be a good, compact, connected, orientable 3-orbifold, and $X$ a 3-orbifold composition. Suppose that there exists an edge OISIBO $\; Y\times[0,1]$ of $X$, the core  $Y$ satisfies that $O_2(X-Y)$ and $O_2(Y)$ are trivial. Then, for any b-continuous map $f:M\to X$, there exists a b-continuous map $g=(|g|,\tilde{g}):M\to X$ which satisfies the following :
\begin{enumerate}
\item[{\bf (i)}] $g$ is C-equivalent to $f$.
\item[{\bf (ii)}] Each component of $g^{-1}(Y)$ is a compact, properly embedded, 2-sided, incompressible 2-suborbifold in $M$.
\item[{\bf (iii)}] For properly chosen product neighborhoods $Y\times [-1,1]$ of $Y=Y\times 0$ in $X$, and $g^{-1}(Y)\times [-1,1]$ of $g^{-1}(Y)=g^{-1}(Y)\times 0$ in $M$, $|g|$ maps each fiber $x\times |[-1,1]|$ homeomorphically to the fiber $|g|(x)\times |[-1,1]|$ for each $x\in |g^{-1}(Y)|$.
\end{enumerate}
\end{theorem}

\begin{proof}
The proof is similar to that of \cite[Theorem 5.5]{finite}.
\end{proof}

\begin{corollary}\label{c-eq-cor}
In Theorem \ref{trans}, if $O_i(X)$ are trivial, $i=2,3$, then $f$ and $g$ are b-homotopic by Remark \ref{c-eq-rem}.
\end{corollary}

\section{Group representations and splittings of groups}

For the contents of the present section we refer to the original paper \cite{C-S}.

An isomorphism from a group $\Pi$ to $\pi_1(G,{\mathcal G})$ is called a {\it splitting\/} of $\Pi$ where $\pi_1(G,{\mathcal G})$ is the fundamental group of a graph of groups $(G,{\mathcal G})$. A splitting is {\it trivial\/} if there exists a vertex group which is isomorphic to the whole $\pi_1(G,{\mathcal G})$.

Let $\Pi$ be a finitely generated group. Take a system of generators $g_1,\dots,g_n$ for $\Pi$. We define a set $R(\Pi)=\{ (\rho(g_1),\dots,\rho(g_n))\;|\; \mbox{a representation}\; \rho:\Pi\rightarrow SL_2(\mathbb{C}) \}$. The points of $R(\Pi)$ correspond to the representations of $\Pi$ in $SL_2(\mathbb{C})$ bijectively, and we often identify such 1-1 corresponding points. For each $g\in\Pi$, we may define a map $\tau_g:R(\Pi)\rightarrow \mathbb{C}$ by $\tau_g(\rho)=\mbox{tr}(\rho(g))$. Let $T$ be the ring generated by all the functions $\tau_g$, $g\in\Pi$. It is finitely generated (\cite[Proposition 1.4.1]{C-S}). Thus we can take and fix $\gamma_1,\dots,\gamma_m\in\Pi$ such that  $\tau_{\gamma_1},\dots,\tau_{\gamma_m}$ generate $T$. With those elements we define a map $t:R(\Pi)\rightarrow \mathbb{C}^m$ by $t(\rho)=(\tau_{\gamma_1}(\rho),\dots,\tau_{\gamma_m}(\rho))$, and set $\mathbb{X}(\Pi)=t(R(\Pi))$. For each $g\in\Pi$, there exists a function $I_g:\mathbb{X}(\Pi)\rightarrow\mathbb{C}$ which maps $t(\rho)$ to $tr(\rho(g))$. It is a regular function.

\begin{theorem}{\rm (}\cite[Theorem 2.1.1]{C-S}{\rm )}\label{CS-2.1.1-th}
If the group $\Pi$ acts without edge inversions on the tree $T$, then $\Pi$ is isomorphic to $\pi_1(T/\Pi,{\mathcal G})$ where $(T/\Pi,{\mathcal G})$ is defined in \cite[pp.123-124]{C-S}.
\end{theorem}

\begin{theorem}{\rm (}\cite[Theorem 2.2.1]{C-S}{\rm )}\label{CS-main-th}
Let $C$ be an affine curve contained in $\mathbb{X}(\Pi)$ and $\tilde{C}$ be a non-singular projective curve uniquely determined by $C$. To each ideal point $\tilde{x}$ of $\tilde{C}$ one can associate a splitting of $\Pi$ with the property that an element $g$ of $\Pi$ is contained in a vertex group if and only if $\tilde{I}_g$ does not have a pole at $\tilde{x}$. Thus, in particular, the splitting is non-trivial.
\end{theorem}

\section{Main Theorem}

\begin{theorem}\label{main1}
Let $M$ be a good, compact, connected, orientable 3-orbifold without non-separating spherical 2-orbifolds. We assume that the fundamental group of each prime component of $M$ is infinite. Suppose that $\pi_1(M)$ has a nontrivial finite splitting. Then there exists a non-empty system of essential 2-suborbifolds $S_1,\dots,S_n\subset M$ such that for each component $Q$ of $\displaystyle M-\bigcup_{i=1}^nS_i$, $\pi_1(Q)$ is contained in a vertex group.
\end{theorem}

\begin{proof}
From the hypotheses, $\pi_1(M)$ is isomorphic to $\pi_1(G,{\mathcal G})$, the fundamental group of a graph of groups $(G,{\mathcal G})$. Along the splitting of $\pi_1(M)$, we construct an orbifold composition (of general type) $X$ as follows:

\underline{Step a} \quad
We may assume that $M$ has no spherical boundary components. Take a base point $y_0\in|M|-\Sigma M$ of $M$.

\underline{Step b} \quad
By the hypotheses, we can take a prime decomposition of $M$, each component of which has an infinite fundamental group. Gluing back the prime components by b-continuous maps along the ballic 3-orbifolds attached in capping off, we obtain an OISIBO $W$ with $\pi_1(W)\cong \pi_1(M)$, each particle of which is an irreducible 3-orbifold.

\underline{Step c} \quad
Let $G_i$, $i\in I$ and $H_j$, $j\in J$ be the vertex groups and the edge groups of $(G,\mathcal{G})$, respectively. Take the covering OISIBO $X_i$ of $W$ associated with each vertex group $G_i$ of $(G,{\mathcal G})$, and the covering OISIBO $Y_j$ of $W$ associated with each edge group $H_j$ of $(G,{\mathcal G})$. If an edge $e_j$ of $G$ has verteces $v_{j_0}$, $v_{j_1}$, then $H_j<G_{j_t}$, $t=0,1$. Thus there exist covering maps $p^j_t:Y_j\to X_{j_t}$ which induce monomorphisms $(p^j_t)_*:H_j\to G_{j_t}$, $t=0,1$.

\underline{Step d} \quad
The system $X=(X_i,Y_j\times [0,1],p^j_t)_{t=0,1,i\in I,j\in J}$ is a desired orbifold composition with $\pi_1(X)\cong\pi_1(T/\pi_1(M),{\mathcal G})\cong\pi_1(M)$. Take a base point $x_0\in |Y_1\times \frac{1}{\; 2\;}|-\Sigma (Y_1\times\frac{1}{\; 2\;})$ of $X$.

\vskip2mm
By Proposition \ref{Oitrivial-prop}, $O_i(X)$ are trivial, $i=1,2,3$. By Proposition \ref{kune-Prop5.3'}, we can construct a b-continuous map $f:M\to X$ which induces an isomorphism $\varphi:\pi_1(M,y_0)\rightarrow \pi_1(X,x_0)$.

Note that the set $J$ is finite. For all $j\in J$, take $f^{-1}(Y_j\times\{\frac{1}{\; 2\;}\})$. By applying Proposition \ref{Oitrivial-prop} for $X-Y_j$, we obtain the fact that $O_2(X-Y_j)$ is trivial. And applying Proposition \ref{kune-Prop5.9} for $Y_j$, we obtain the fact that $O_2(Y_j)$ is trivial. We have already shown that $O_i(X)$'s are trivial, $i=2,3$. With modifications through b-homotopies if necessary, by Theorem \ref{trans} and Corollary \ref{c-eq-cor}, each component of $f^{-1}(\bigcup_{j\in J}(Y_j\times\{\frac{1}{\; 2\;}\}))$ is a compact, properly embedded, 2-sided, incompressible 2-suborbifold in $M$. If one of such components is boundary parallel, we can reduce the number of components by modifications through b-homotopies. Note that we can make those modifications fixing on some neighborhood of each component of $\displaystyle\bigcup_{j\in J}f^{-1}(Y_j)$, which is already imcompressible. After those modifications, we obtain a system of essential 2-suborbifolds $S_1,\cdots,S_n$ as (not necessarily all) components of $f^{-1}(Y_j\times\{\frac{1}{\; 2\;}\})$, $j\in J$.

By the construction of $S_1,\cdots,S_n$, for each component $Q$ of $M-\bigcup_{j=1}^nS_j$, $\pi_1(Q)$ is contained in a vertex group of $(G,{\mathcal G})$.
\end{proof}

\begin{corollary}\label{main-cor1}
Let $M$ be a good, compact, connected, orientable 3-orbifold without non-separating spherical 2-orbifolds. We assume that the fundamental group of each prime component of $M$ is infinite. Suppose that $\pi_1(M)$ acts on a simplicial tree $T$ nontrivially without edge inversions such that $T/\pi_1(M)$ is finite. Then there exists a non-empty system of essential 2-suborbifolds $S_1,\dots,S_n\subset M$ such that for each component $Q$ of $\displaystyle M-\bigcup_{i=1}^nS_i$, $\pi_1(Q)$ fixes a vertex of the tree.
\end{corollary}

\begin{proof}
By Theorems \ref{CS-2.1.1-th} and \ref{main1}.
\end{proof}

\begin{corollary}\label{main-cor2}
Let $M$ be a good, compact, connected, orientable 3-orbifold without non-separating spherical 2-orbifolds. We assume that the fundamental group of each prime component of $M$ is infinite. Let $C$ be an affine curve contained in $\mathbb{X}(\pi_1(M))$. To each idel point $\tilde{x}$ of $\tilde{C}$ one can associate a splitting of $\pi_1(M)$ with the property that an element $g$ of $\pi_1(M)$ is contained in a vertex group if and only if $\tilde{I}_g$ does not have a pole at $\tilde{x}$. Then there exists a non-empty system of essential 2-suborbifolds $S_1,\dots,S_n\subset M$ such that for each component $Q$ of $\displaystyle M-\bigcup_{i=1}^nS_i$, $\pi_1(Q)$ is contained in a vertex group.
\end{corollary}

\begin{proof}
By Theorems \ref{CS-main-th} and \ref{main1}.
\end{proof}

\end{document}